\documentclass{amsart}

\newtheorem{proposition}{Proposition}[section]
\newtheorem{lemma}{Lemma}[section]
\newtheorem{corollary}{Corollary}[section]
\theoremstyle{definition}

\newtheorem{example}{Example}[section]

\newtheorem{conjecture}{Conjecture}[section]
\newtheorem{remark}{Remark}[section]

\usepackage{amscd}
\usepackage{amsthm}
\usepackage{amsfonts}
\usepackage{amsbsy}
\usepackage{amsmath}
\usepackage{amssymb}
\usepackage{color}	
\usepackage{enumerate}
\usepackage{epsfig}
\usepackage[mathscr]{eucal}
\usepackage{fancyhdr}
\usepackage{mathrsfs}
\usepackage{graphicx}
\usepackage{hyperref}
\hypersetup{colorlinks=false, urlcolor=black}
\usepackage{latexsym}
\usepackage{listings}
\usepackage{longtable}
\usepackage{nomencl}
\usepackage[numbers]{natbib}
\usepackage{psfrag}
\usepackage{setspace}
\usepackage{tikz}
\usepackage{url}
\usepackage{undertilde}
\usepackage{verbatim}
\usepackage{graphics,epstopdf}
\usepackage[all]{xy}
\newcommand{\Z}{\mathbb{Z}}

\begin{document}

\title{Twin primes and the 3-sphere}
\author{Samuel A. Hambleton}
\subjclass[2020]{Primary 11L17, Secondary 11P32}
\date{7 July, 2024}

\address{School of Science, Technology and Engineering, University of the Sunshine Coast, Maroochydore DC, Australia and \\
School of Mathematics and Physics, The University of Queensland, Australia}

\begin{abstract}
We investigate the group of points of the 3-sphere modulo a prime, point out connections to other known groups and the Chebyshev polynomials, and show that there is an infinite series which converges if and only if there are finitely many pairs of twin primes. Hence to prove that the series diverges is to prove the twin prime conjecture. 
\end{abstract}

\keywords{3-sphere, special unitary group, special orthogonal group, quaternionic multiplication, four square theorem, Wilson's theorem, primality criteria, twin prime conjecture}

\maketitle

\Large 

\section{Introduction}

There are a handful of statements known to give necessary and sufficient conditions for the positive integer $p$ to be prime. Perhaps the most well known is Wilson's theorem \cite[pp. 67]{Erdos}, the positive integer $p$ is prime if and only if $$(p-1)! \equiv -1 \pmod{p} .$$ Of theoretical value we form a similar characterization of primes and twin primes but with different computational limitations than calculating the factorial of an integer. The $3$-sphere 
\begin{equation}\label{3sphere}
\mathcal{S} : x_1^2 + x_2^2 + x_3^2 + x_4^2 = 1  
\end{equation}
is of great importance in relation to quaternions, the wedge product, Lie algebra, and Particle Physics \cite{Girard}. Lagrange showed that every positive integer can be expressed as the sum of the squares of four integers and Jacobi counted the number of representations of the positive integer $n$ as the sum of four integer squares. Schmutz \cite{Schmutz} showed that the unit $N$-sphere with center the origin has a dense set of rational points.  

In this article we consider groups of points satisfying \ref{3sphere} modulo $n$ and point out that this forms a non-abelian group. We show that $p$ is an odd prime if and only if the order of the group $\mathcal{S}(\mathbb{Z}/p) $ is equal to $p^3 - p$ in Corollary \ref{maincor} and give a criterion for $p$ and $p+2$ to be twin primes in Corollary \ref{twins}. We exhibit an infinite series which converges if and only if the twin prime conjecture is false in Proposition \ref{propseries}. In the final section we conclude with some observations on relationships between $\mathcal{S}(\mathbb{Z}/p) $ are other finite groups.   

Euler knew that if two positive integers can be written as the sum of four squares, then their product can also be written as the sum of four squares \cite[pp. 210]{Erdos}. In fact, Equation \eqref{grouplaw} below coincides with Euler's four square identity \cite[pp. 193]{Euler}. This suggests that the $3$-sphere has an algebraic structure, and the structure is essentially known as quaternionic multiplication. We begin with a proposition contextualizing the group structure of the $3$-sphere over the ring commutative ring with unity $\mathcal{R}$. However, the simplest way to understand and work with the group $\mathcal{S}(\mathcal{R})$ is to do so via the image of an injective map from $\mathcal{S}(\mathcal{R})$ to the special orthogonal group of order $4$. In other words compute $X \oplus Y$ by multiplying the matrices $\phi (X) \phi (Y)$ defined in \ref{firstprop} below. This works whether or not $\sqrt{-1}$ is an element of $\mathcal{R}$. 

\begin{proposition}\label{firstprop}
Let 
\begin{align*}
X & = \left(
\begin{array}{cccc}
 x_1 & x_2 & x_3 & x_4 \\
\end{array}
\right) , & Y & = \left(
\begin{array}{cccc}
 y_1 & y_2 & y_3 & y_4 \\
\end{array}
\right) 
\end{align*}
be points of the $3$-sphere $\mathcal{S}$ with coordinates in the commutative ring with unity $\mathcal{R}$ such that the imaginary number $i \not\in \mathcal{R}$. Define the maps 
\begin{align*}
\theta & : \mathcal{S}( \mathcal{R} ) \longrightarrow \text{SU}_2 \left( \mathcal{R} \right) , & \theta (X) & = \left(
\begin{array}{cc}
 x_1 + x_2 i & - x_3 + x_4 i \\
 x_3 + x_4 i & x_1 - x_2 i \\
\end{array}
\right) , \\
\phi & : \mathcal{S}( \mathcal{R} ) \longrightarrow \text{SO}_4( \mathcal{R} ) , & \phi ( X ) & = \left(
\begin{array}{cccc}
 x_1 & x_2 & x_3 & x_4 \\
 -x_2 & x_1 & x_4 & -x_3 \\
 -x_3 & -x_4 & x_1 & x_2 \\
 -x_4 & x_3 & -x_2 & x_1 \\
\end{array}
\right) , \\
\end{align*}
where $\text{SU}_2 \left( \mathcal{R} \right) $ denotes the special unitary group of degree $2$ with entries in the ring $\mathcal{R}[i]$ and $\text{SO}_4( \mathcal{R} )$ is the special orthogonal group of order $4$ with entries in $\mathcal{R}$. Then: 
\begin{enumerate}
\item $\theta$ is a bijection. Hence $\langle \mathcal{S}(\mathcal{R}) , \oplus \rangle $ is a non-abelian group, where 
\begin{align*}
\oplus & : \mathcal{S}(\mathcal{R}) \times \mathcal{S}(\mathcal{R}) \longrightarrow \mathcal{S}(\mathcal{R}) , & X \oplus Y & = \theta^{-1} \left( \theta (X) \theta (Y) \right) .
\end{align*} 
\item $\phi$ is an injective group homomorphism. 
\item If $\mathcal{R} = \mathbb{Z}/n$, where $n$ is a positive integer, then there is a short exact sequence 
\begin{equation*}
1  \longrightarrow \mathcal{S}\left(\mathbb{Z}/n\right) \stackrel{f}{\longrightarrow} U_2 \left( \mathbb{Z}/n \right) \stackrel{g}{\longrightarrow} \left(\mathbb{Z}/n\right)^{\times } \longrightarrow 1  
\end{equation*}
of group homomorphisms, where $U_2 \left( \mathbb{Z}/n \right)$ is the unitary group.
\end{enumerate}
\end{proposition}

\begin{proof}
Let $x_1,x_2,x_3,x_4 \in \mathcal{R}$ satisfy $x_1^2 + x_2^2 + x_3^2 + x_4^2 = 1$, let \\
$X = \left(
\begin{array}{cccc}
 x_1 & x_2 & x_3 & x_4 \\
\end{array}
\right)$, and let $\alpha = x_1 + x_2 i, \beta = x_3 + x_4 i \in \mathcal{R}[i]$. Then $\theta (X) = \left(
\begin{array}{cc}
 \alpha & - \overline{\beta } \\
 \beta & \overline{\alpha } \\
\end{array}
\right) \in \text{SU}_2 \left( \mathcal{R} \right)$ so $\theta $ is well defined. To show that $\theta $ is injective, assume that $\theta (X) = \theta (Y)$. Then $x_1 + x_2 i = y_1 + y_2 i$ and $x_3 + x_4 i = y_3 + y_4 i$ so that $X = Y$ since $i \not\in \mathcal{R}$. To show that $\theta $ is surjective, let $A = \left(
\begin{array}{cc}
 \alpha & - \overline{\beta } \\
 \beta & \overline{\alpha } \\
\end{array}
\right) \in \text{SU}_2 \left( \mathcal{R} \right)$ with $\alpha , \beta \in \mathcal{R}[i]$. There exists $x_1, x_2, x_3, x_4 \in \mathcal{R}$ such that $\alpha = x_1 + x_2 i$, $\beta = x_3 + x_4 i$ and since $\alpha \overline{\alpha } + \beta \overline{\beta } = 1$, we have $x_1^2 + x_2^2 + x_3^2 + x_4^2 = 1$. Hence $\theta $ is a bijection, $\theta^{-1}$ exists, and by transport of the structure of the group $ \text{SU}_2 \left( \mathcal{R} \right) $, $\langle \mathcal{S}(\mathcal{R}) , \oplus \rangle $ is a non-abelian group, where 
\begin{align*}
\oplus & : \mathcal{S}(\mathcal{R}) \times \mathcal{S}(\mathcal{R}) \longrightarrow \mathcal{S}(\mathcal{R}) , & X \oplus Y & = \theta^{-1} \left( \theta (X) \theta (Y) \right) .
\end{align*}
We have 
\begin{equation}\label{grouplaw}
X \oplus Y = \left( x_1 y_1 - \utilde{X} \cdot \utilde{Y} \ \mid \ \utilde{X} y_1 + x_1 \utilde{Y} - \utilde{X} \times \utilde{Y} \right) , 
\end{equation}
where $\utilde{X}$ is obtained by omitting the first coordinate of $X$, $\cdot $ is the scalar product and $\times $ is the cross product.

The map $\phi $ is well defined since if $X \in \mathcal{S}(\mathcal{R})$, then $\phi (X)$ is orthogonal with $$\phi (X) \phi (X)^T = \left( x_1^2 + x_2^2 + x_3^2 + x_4^2 \right) I = I = \phi (X)^T \phi (X) .$$  
It is easy to verify that $\phi (X) \phi (Y) = \phi (Z)$, where $Z = X \oplus Y$ given by \eqref{grouplaw}. It follows that $\phi $ is a group homomorphism. The kernel of $\phi $ is $\left\{ O \right\}$, where $O =  \left(
\begin{array}{cccc}
 1 & 0 & 0 & 0 \\
\end{array}
\right)$ is the identity of $\mathcal{S}(\mathcal{R})$ so $\phi $ is injective. 

Consider the unitary group $U_2 \left( \left(\mathbb{Z}/n \right) \right)$,
\begin{equation*}
U_2 \left( \left(\mathbb{Z}/n \right) \right) = \left\{ \text{ invertible 2 by 2 } \left(
\begin{array}{cc}
 \alpha & - \overline{\beta } \\
 \beta & \overline{\alpha } \\
\end{array}
\right) \ : \ \alpha , \beta \in \left(\mathbb{Z}/n \right) \right\} . 
\end{equation*}
We show that there is a short exact sequence
\begin{equation*}
1  \longrightarrow \mathcal{S}\left(\mathbb{Z}/n\right) \stackrel{f}{\longrightarrow} U_2 \left( \mathbb{Z}/n \right) \stackrel{g}{\longrightarrow} \left(\mathbb{Z}/n\right)^{\times } \longrightarrow 1  
\end{equation*}
of group homomorphisms, where $$f \left(
\begin{array}{cccc}
 x_1 & x_2 & x_3 & x_4 \\
\end{array}
\right) = \left(
\begin{array}{cc}
 x_1 + x_2 i & - x_3 + x_4 i \\
 x_3 + x_4 i & x_1 - x_2 i \\
\end{array}
\right) ,$$ and $g \left( \left(
\begin{array}{cc}
 \alpha & - \overline{\beta } \\
 \beta & \overline{\alpha } \\
\end{array}
\right) \right) = \alpha \overline{\alpha } + \beta \overline{\beta }$, the determinant of the matrix modulo $n$. This can be viewed as similar to the well known short exact sequence of Lie groups $$ 1 \longrightarrow SU(N) \longrightarrow U(N) \longrightarrow U(1) \longrightarrow 1,$$ where the $N \times N$ matrices are $2 \times 2$ in this case. The group operation was defined such that $f$ is a homomorphism and $g$ is a homomorphism since determinants are multiplicative. $g$ is surjective by Lagrange's four square theorem. The kernel of $g$ is the set of all matrices in $U_2 \left( \left(\mathbb{Z}/n \right) \right)$ of determinant $1$. This coincides with the image of the map $f$.
\end{proof}

It is easy to use matrices in $\phi [ \mathcal{S}(\mathcal{R})]$ to calculate in the group $\mathcal{S}(\mathcal{R})$. The transpose of the matrix $\phi ( X )$ is the same matrix discovered by Elfrinkhof \cite{Elfrinkhof} in association with the factorization of a rotation matrix. A number of useful identities on quaternionic multiplication are relevant to these groups can be found in \cite[Chp. 8]{Gallier}. Regardless of whether $i \in \mathcal{R}$, Equation \eqref{grouplaw} defines the group law $\oplus$. 

Equation \eqref{grouplaw} hints at the existence of some geometric interpretation of the group law $\oplus$ for $\mathcal{S}$ similar to the construction of the group law for elliptic curves or Pell conics \cite{Lemm}. See Figure \ref{uncirc} and Remark \ref{conics}.

\begin{figure}[h]
\centering
  \includegraphics[width={8.0cm}]{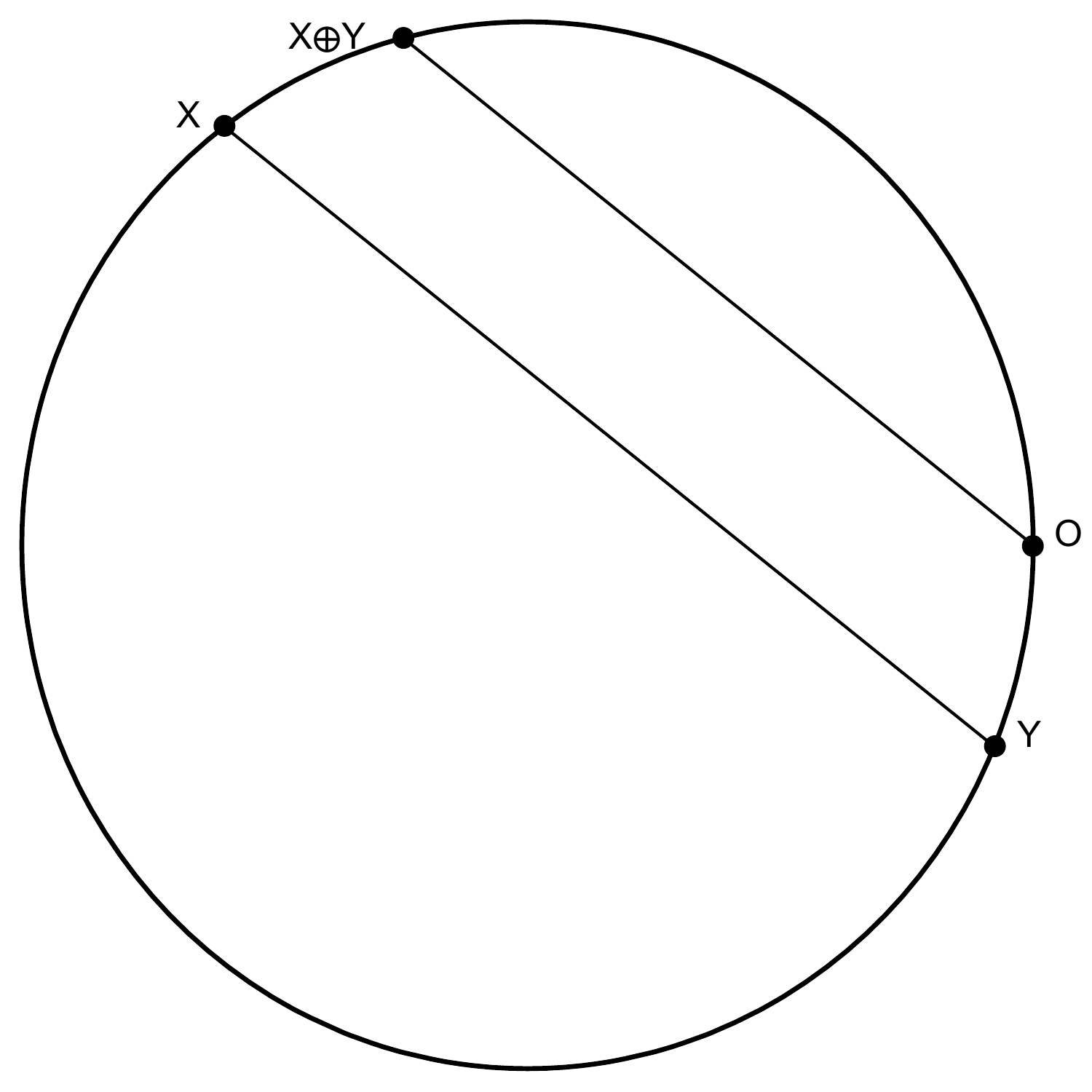}
  \caption{The group law on the unit circle.}
  \label{uncirc}
\end{figure} 

It is interesting to note that: 
\begin{remark}
Like the circle, the $3$-sphere without the rational sphere $\mathcal{S}_2 \ : \ x_1^2 + x_2^2 + x_4^2 = 1; \ x_3 = 0$ is rational. There is a birational map 
\begin{align*}
\mathbb{Q}^{\ast} \times \mathbb{Q}^2 & \longrightarrow \mathcal{S}(\mathbb{Q}) - \mathcal{S}_2 (\mathbb{Q}), 
& (t, u, v) & \longmapsto \left(
\begin{array}{cccc}
 \frac{w - 2 t^2}{w} & \frac{2 t u}{w} & \frac{2 t}{w} & \frac{2 t v}{w} \\
\end{array}
\right) ,
\end{align*}
where $w = 1 + t^2 + u^2 + v^2$ and $\mathbb{Q}^{\ast} = \mathbb{Q} - \{ 0 \}$ so $t \not= 0$. \\
If $\left(
\begin{array}{cccc}
 x_1 & x_2 & x_3 & x_4 \\
\end{array}
\right) \in \mathcal{S}(\mathbb{Q}) - \mathcal{S}_2 (\mathbb{Q})$, then 
\begin{align*}
t & = \frac{1-x_1}{x_3} , & u & = \frac{x_2}{x_3} , & v & = \frac{x_4}{x_3} , & w & = \frac{2 \left( 1 - x_1 \right) }{x_3^2} .
\end{align*}
If $\left(
\begin{array}{ccc}
 t & u & v \\
\end{array}
\right) \in \mathcal{S}_2 (\mathbb{Q})$, then $\left(
\begin{array}{cccc}
 1 - t^2 & t u & t & t v \\
\end{array}
\right) \in \mathcal{S}_2 (\mathbb{Q})$. 
\end{remark}

\section{Primes and the group modulo $n$}

Corollary \ref{maincor} below gives necessary and sufficient conditions for primality via the $3$-sphere. To prove this, we must first count the number of points of the $3$-sphere modulo $n$. Throughout this section we will let \\
$R_4(n) = \# \mathcal{S} (\mathbb{Z} / n)$ be the number of points satisfying $$x_1^2 + x_2^2 + x_3^2 + x_4^2 \equiv 1 \pmod{n} .$$

The following lemma is a consequence of a result due to Weil who counted points modulo primes for general diagonal equations. See \cite [pp. 142]{IreRosen} for example.

\begin{lemma}\label{Weil}
If $p$ is prime, then $$R_4 ( p ) = \left\{ \begin{array}{cc}
  2^3 & \text{if } p \text{ is even,} \\
 p^3 - p & \text{if } p \text{ is odd.} \\
\end{array} \right. .$$
\end{lemma}

Next we count the points of the $3$-sphere over a finite field. 

\begin{lemma}\label{Hensel}
If $q = p^{e}$ is the power of a prime $p$, then $$R_4( q ) = \left\{ \begin{array}{cc}
  q^3 & \text{if } p \text{ is even,} \\
 q^3 \left( 1 - p^{-2} \right) & \text{if } p \text{ is odd.} \\
\end{array} \right. .$$
\end{lemma}

\begin{proof}
This follows from Lemma \ref{Weil} together with Hensel lifting and induction. Let $p$ be an odd prime. By Lemma \ref{Weil} the number of solutions to $x_1^2 + x_2^2 + x_3^2 + x_4^2 \equiv 1 \pmod{p}$ is $p^{3} - p$. Assume that the number of solutions to $x_1^2 + x_2^2 + x_3^2 + x_4^2 \equiv 1 \pmod{p^e}$ is $p^{3 e} - p^{3 e-2}$. Then by Hensel's lemma, for every $\left(
\begin{array}{cccc}
 x_1 & x_2 & x_3 & x_4 \\
\end{array}
\right)$ satisfying \\
$x_1^2+x_2^2+x_3^2+x_4^2 \equiv 1 \pmod{p^e}$, there are $p^3$ solutions modulo $p^{e+1}$ which reduce to the solution $\left(
\begin{array}{cccc}
 x_1 & x_2 & x_3 & x_4 \\
\end{array}
\right)$ modulo $p^{e}$. Hence the number of solutions to
$x_1^2 + x_2^2 + x_3^2 + x_4^2 \equiv 1 \pmod{p^{e+1}}$ is $$p^3 \left( p^{3 e} - p^{3 e-2} \right) = p^{3 (e + 1)} - p^{3 (e+1) - 2} .$$ The result follows by the principle of induction. We proceed similarly for $q = 2^e$.  
\end{proof}

The function $R_4(n)$ is multiplicative as the following result demonstrates.

\begin{lemma}\label{multiplicative}
If $\gcd \left( n_1, n_2 \right) = 1$, then $R_4 \left( n_1 \right) R_4 \left( n_2 \right) = R_4 \left( n_1 n_2 \right) $.
\end{lemma}

\begin{proof}
By the Chinese remainder theorem there is a bijection \\
$\rho : \mathcal{S}(\mathbb{Z}/n_1) \times \mathcal{S}(\mathbb{Z}/n_2) \longrightarrow \mathcal{S}(\mathbb{Z}/(n_1 n_2) )$ by $$\left( X \pmod{n_1} , \ X \pmod{n_2} \right) \longmapsto X \pmod{n_1 n_2} .$$ 
There are $R_4 \left( n_1 \right) R_4 \left( n_2 \right)$ possible pairs of points \\
$\left( X \pmod{n_1} , \ X \pmod{n_2} \right)$. From the bijection $\rho$ it follows that \\
$R_4 \left( n_1 \right) R_4 \left( n_2 \right) = R_4 \left( n_1 n_2 \right) $.
\end{proof}

We are now able to state the order of the group $\mathcal{S} (\mathbb{Z} / n)$ for any positive integer $n$.

\begin{proposition}\label{mainres}
Let $R_4(n) = \# \mathcal{S} (\mathbb{Z} / n)$ be the number of points satisfying $$x_1^2 + x_2^2 + x_3^2 + x_4^2 \equiv 1 \pmod{n} .$$ Then 
\begin{equation*}
R_4(n) = n^3 \prod_{\substack{p \mid n  \\ p \text{ odd }}} \left( 1 - p^{-2} \right) .
\end{equation*} 
\end{proposition}

\begin{proof}
If $\gcd \left( n_1, n_2 \right) = 1$, then $R_4 \left( n_1 \right) R_4 \left( n_2 \right) = R_4 \left( n_1 n_2 \right)$ by Lemma \ref{multiplicative}. Let $n = 2^{e_0} p_1^{e_1} p_2^{e_2} \dots p_m^{e_m}$, where the $p_1, p_2 \dots p_m$ are odd prime divisors of $n$ and $e_0 $ can be $0$. Then 
\begin{eqnarray*}
R_4(n) & = & R_4 \left( 2^{e_0} \right) R_4 \left( p_1^{e_1} \right) \dots R_4 \left( p_m^{e_m} \right) , \\
       & = & 2^{3 e_0} p_1^{3 e_1} \left( 1 - p_1^{-2} \right) \dots p_m^{3 e_m} \left( 1 - p_m^{-2} \right) , \\
       & = & n^3 \prod_{\substack{p \mid n  \\ p \text{ odd }}} \left( 1 - p^{-2} \right) .
\end{eqnarray*} 
\end{proof}

\begin{corollary}\label{maincor}
$n$ is an odd prime if and only if the number of solutions to the congruence $$x_1^2 + x_3^2 + x_3^2 + x_4^2 \equiv 1 \pmod{n}$$ is equal to $n^3 - n$; equivalently, if and only if $$\prod_{\substack{p \mid n  \\ p \text{ odd }}} \left( 1 - \frac{1}{p^2} \right)^{-1} = \left( 1 - \frac{1}{n^2}\right)^{-1} .$$
\end{corollary}

\begin{proof}
If $n$ is an odd prime, then by Proposition \ref{mainres}, $$\# \mathcal{S} \left( \mathbb{Z}/n \right) = n^3 \left( 1 - n^{-2} \right) = n^3 - n .$$ If $n$ is composite, let $n = 2^{e_0} p_1^{e_1} \dots p_m^{e_m}$ be the factorization of $n$, where $e_0 \geq 0$ and $p_1, p_2, \dots , p_m$ are precisely the prime divisors of $n$. Suppose $$n^3 \left( 1 - p_1^{-2} \right) \left( 1 - p_2^{-2} \right) \dots \left( 1 - p_m^{-2} \right) = n^3 - n .$$ Then $$2^{2 e_0} p_1^{2 e_1-2} \dots p_m^{2 e_m-2} \left( p_1^{2} - 1\right) \left( p_2^{2} - 1 \right) \dots \left( p_m^{2} - 1 \right) = n^{2} -1 .$$ Clearly $n$ cannot be even so $e_0 = 0$. Furthermore, reduction modulo $p_j$ where $p_j$ is an odd prime divisor of $n$ gives the contradiction \\
$0 \equiv -1 \pmod{p_j}$ when $2 e_j - 2 > 0$. It follows that $$n = p_1 p_2 \dots p_m $$ and $$ \left( p_1^{2} - 1\right) \left( p_2^{2} - 1 \right) \dots \left( p_m^{2} - 1 \right) = p_1^2 p_2^2 \dots p_m^2 -1 .$$ Since $p_m > p_{m-1} > \dots p_1 > 2$, $$ \left( p_1^{2} - 1\right) \left( p_2^{2} - 1 \right) \dots \left( p_m^{2} - 1 \right) < p_1^2 p_2^2 \dots p_m^2 -1 ,$$ a contradiction when $m > 1$. Hence we must have $m = 1$ and $n$ is prime. 

For the equivalent statement, observe that if $n$ is an odd positive composite integer, then $$n^2 \prod_{\substack{p \mid n  \\ p \text{ odd }}} \left( 1 - p^{-2} \right) < n^2 - 1 $$ so $R_4(n) < n^3 - n$ and $$\prod_{\substack{p \mid n  \\ p \text{ odd }}} \left( 1 - \frac{1}{p^2} \right)^{-1} > \left( 1 - \frac{1}{n^2}\right)^{-1} .$$ If $n$ is prime, then we have equality. 	  
\end{proof}

\section{Twin primes}

It is thought, according to the twin prime conjecture, that there are infinitely many pairs of primes whose difference is $2$, known as twin primes. Brun \cite{Brun} proved that the sum of the reciprocals of the twin primes, $$\frac{1}{3} + \frac{1}{5} + \frac{1}{5} + \frac{1}{7} + \frac{1}{11} + \frac{1}{13} \approx 1.90216058$$ converges. In this section we prove that there is an infinite series which converges to $N + \tau (s)$ if and only if there are exactly $N$ pairs of twin primes, where $\tau (s)$ is a positive constant depending on $s$. 

Clement \cite{Clement} proved that $n$ and $n+2$ with $n > 1$ are both prime if and only if $$4 ((n-1)! + 1) + n \equiv 0 \pmod{n(n+2)} .$$ As a consequence of Corollary \ref{maincor} we have the following statement about twin primes similar to Clement's result, which we will use to write our infinite series \eqref{mainser}.

\begin{corollary}\label{twins}
The odd positive integers $n$ and $n+2$ are twin primes if and only if 
\begin{equation}\label{twineqn}
\prod_{\substack{p \mid n(n+2)  \\ p \text{ odd }}} \left( 1 - \frac{1}{p^{2}} \right)^{-1} = \frac{n^2 (n+2)^2}{(n-1) (n+1)^2 (n+3)} . 
\end{equation}
\end{corollary}  

\begin{proof}
We prove that the odd positive integers $p$ and $p+2$ are twin primes if and only if there are exactly $p^6+6 p^5+10 p^4-11 p^2-6 p $ solutions to the congruence $$x_1^2 + x_2^2 + x_3^2 + x_4^2 \equiv 1 \pmod{p^2 + 2 p} .$$

If $p$ and $p+2$ are twin primes, then by Corollary \ref{maincor}, the number of solutions to the congruence $$x_1^2 + x_2^2 + x_3^2 + x_4^2 \equiv 1 \pmod{p(p+2)}$$ is equal to $\left( p^3 - p \right) \left( (p+2)^3 - (p+2) \right)$. 

If $n$ is an odd positive composite integer, then $R_4(n) < n^3 - n$. Hence if $p$ or $p+2$ is composite, then the number of solutions to the congruence $$x_1^2 + x_2^2 + x_3^2 + x_4^2 \equiv 1 \pmod{p(p+2)}$$ is less than $\left( p^3 - p \right) \left( (p+2)^3 - (p+2) \right)$, which is equal to \\
$p^6 + 6 p^5 + 10 p^4 - 11 p^2 - 6 p $. The statement on \eqref{twineqn} follows. 
\end{proof}

Reflecting on Corollary \ref{twins}, we have the following proposition which shows that the twin prime conjecture is equivalent to the statement that $\omega (s)$ diverges for some real value of $s > 2$, say $s = 3$, where $\omega (s)$ is the series given by \eqref{mainser}. 

\begin{proposition}\label{propseries}
Let 
\large 
\begin{align*}
 E(x) & = \left( 1 - \frac{1}{x^2} \right)^{-1} , & A_n & = \left(
\begin{array}{cc}
 \displaystyle\prod_{p \mid 2 n + 1} E(p) & E(2 n + 1)  \\
 E(2 n + 3)  & \displaystyle\prod_{p \mid 2 n + 3} E(p)  \\
\end{array}
\right) ,
\end{align*}
\Large 
where the products in $A_n$ is taken over prime divisors. If $s > 2$ is a real number, then the series 
\begin{equation}\label{mainser}
\omega (s) = \sum_{n=1}^{\infty} \frac{1}{1 + \left| A_n \right| n^s}
\end{equation}
converges to $N + \tau (s)$ if and only if there are finitely many pairs of twins primes $N$, where $\tau (s)$ is a positive constant depending on $s$. 
\end{proposition}

\begin{proof}
Let
\begin{align*}
\tau (s, m) & = \sum_{n=1}^{m} F(s, n) , & F(s, n) & = \left\{ \begin{array}{cc}
  0 & \text{if } \left| A_n \right| = 0, \\
 \left( 1 + \left| A_n \right| n^s \right)^{-1} & \text{otherwise.} \\
 \end{array} \right. 
\end{align*}
The least positive value $\left| A_n \right|$ can attain occurs when $2 n + 1$ and $\sqrt{2 n + 3}$ are both prime. Let $g(n) = (2 n+1)(2 n+3)$ and $h(n) = \left( g(n) - 1\right)^2 - 4$. It follows that if $\left| A_n \right| \not= 0$, then\large $$\frac{\pi^2}{6} - 1 > \left| A_n \right| \geq E(2 n + 1) E \left( \sqrt{2 n + 3} \right) - \frac{g^2(n)}{h(n)} ,$$ \Large and we have
\begin{equation*}
F(s, n) \leq \frac{h(n)}{ h(n) + (2 n + 1)^2 (2 n + 3) n^s} < \frac{s}{n^{s-1}} . 
\end{equation*}
When $s > 2$ the series $ \displaystyle\sum_{n=1}^{\infty } F(s, n) < s \zeta (s-1)$ converges to \\
$\tau (s) = \displaystyle\lim_{m \longrightarrow \infty } \tau (s, m)$ by the $p$-test and the comparison test, where $\zeta (s)$ is the Riemann zeta function. By Corollary \ref{twins} $$\sum_{n=1}^{m} \frac{1}{1 + \left| A_n \right| n^s} = \pi_2(2 m + 3) + \tau (s, m) ,$$ where $\pi_2(x)$ is the number of pairs of twin primes $(p, q)$ satisfying $q \leq x$. It follows that if there are a total of $N$ pairs of twin primes, then the series \eqref{mainser} converges to $N + \tau (s)$. If there are infinitely many pairs of twin primes, then the series \eqref{mainser} adds the summand $1$ infinitely many times and hence the series diverges.  
\end{proof}

$\pi_2(x)$ was considered by Hardy and Littlewood \cite[pp. 32]{HL}, where it was estimated that
\large 
\begin{align*}
\pi_2 (x) & \approx \frac{2 C_2 x}{\log^2 (x)} = 2 C_2 \int_2^{n} \frac{dt}{\log^2(t) }, & C_2 & = \prod_{p \geq 3} \left( 1 - \frac{1}{(p-1)^2} \right) \approx 0.66016.
\end{align*}
\Large 
Brent \cite{Brent} computed $\pi_2(x)$ for $x < 8 \times 10^{10}$. For an introduction to computation with twin primes see \cite{SebahGourdon}. Calculations in the table below show that $\tau (3) \approx 0.474004103627$. 

\normalsize 
\begin{center}
\begin{tabular}{|c | c | c | c | c|} 
 \hline
 $\log_{10}(m)$ & $\pi_2 (2 m + 3)$ & $\tau (1, m)$ & $\tau (2, m)$ & $\tau (3, m)$  \\ [0.5ex] 
 \hline\hline
 1 & \hfill 4 & \hfill 3.225 & 1.238501511411617 & 0.424789649215940 \\  
 \hline
 2 & \hfill 15 & \hfill 18.619 & 2.088852995430603 & 0.472943322728998 \\ 
 \hline
 3 & \hfill 61 & \hfill 65.555 & 2.305530261242241 & 0.473970392946628 \\ 
 \hline
 4 & \hfill 342 & \hfill 229.208 & 2.377134528816283 & 0.474002620057820 \\ 
 \hline
 5 & \hfill 2160 & \hfill 982.657 & 2.409562315521043 & 0.474004021326231 \\ 
 \hline
 6 & \hfill 14871 & \hfill 5166.336 & 2.427525024274653 & 0.474004098679028 \\ 
 \hline
 7 & \hfill 107407 & \hfill 31095.487 & 2.438465311832159 & 0.474004103304255 \\ 
 \hline
 8 & \hfill 813371 & \hfill 203071.953 & 2.445670220381749 & 0.474004103605379 \\ 
 \hline
 9 & \hfill 6388041 & \hfill 1397233.895 & 2.450652571972753 & 0.474004103626066 \\ 
 \hline
 10 & \hfill 51509099 & 10016194.267 & 2.454233705593651 & 0.474004103627544 \\ 
  \hline
\end{tabular}
\end{center}
\Large

\vfill

\newpage

While it is easy to prove that $\tau (1, m)$ diverges, Figure \ref{tau2m} shows a close approximation 
\begin{equation}\label{tau2mod}
\tau \left( 2, 10^x \right) \approx 2.47299 - \frac{1.63688}{x^2 + 0.325582} ,
\end{equation}
suggesting that $\tau (2, m)$ converges. 
\begin{figure}[h]
\centering
  \includegraphics[width={8.0cm}]{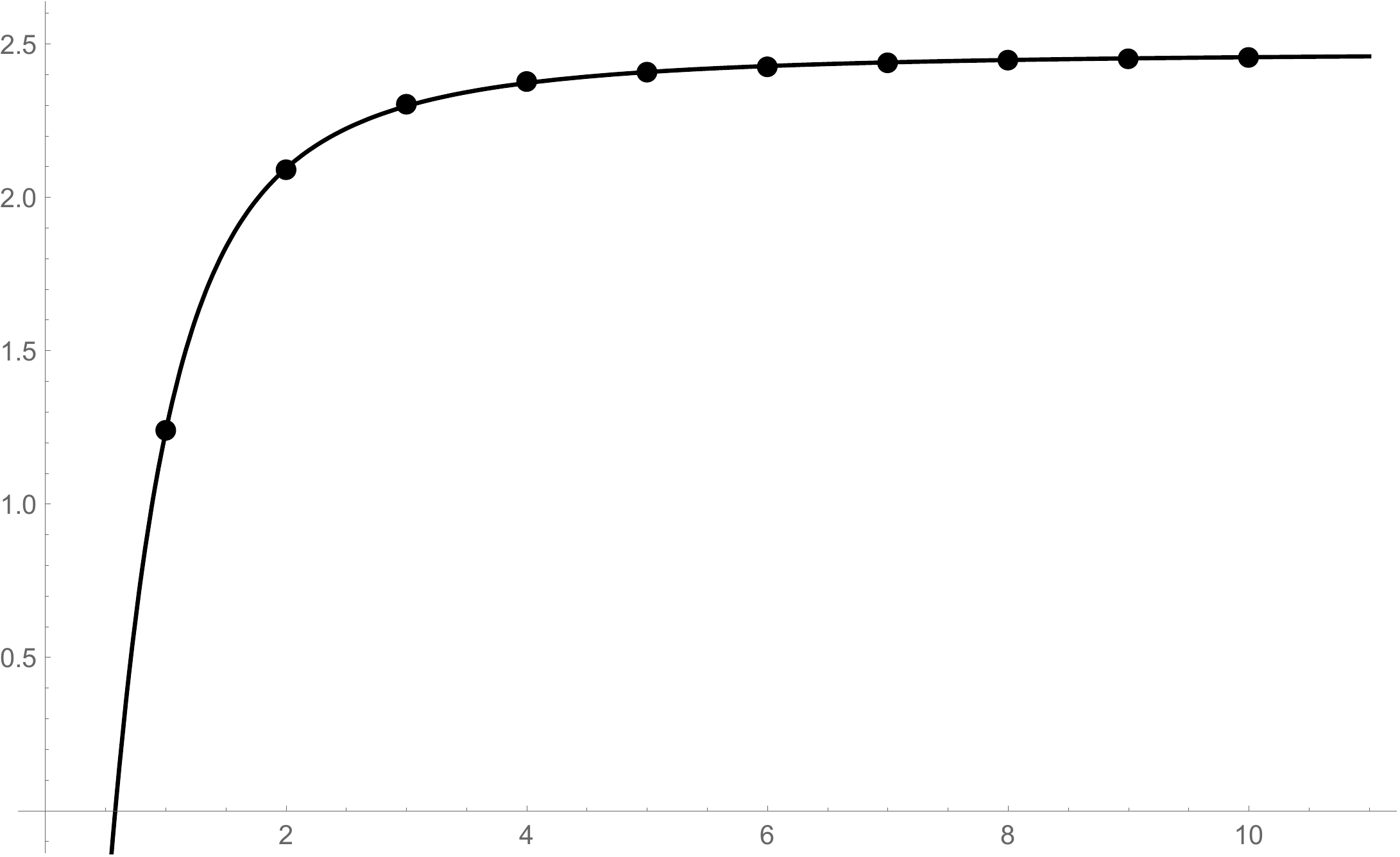}
  \caption{The model of $\tau (2, m)$ given by \eqref{tau2mod}.}
  \label{tau2m}
\end{figure} 

\begin{conjecture} \ 
\begin{enumerate}
\item[{\bf{(a)}}] For all positive integers $m \geq 10$, 
\begin{equation}\label{conject2}
 1 + \tau (3) < \sum_{n=m}^{2 m} \frac{1}{1 + \left| A_n \right| n^3} .
 \end{equation}
\item[{\bf{(b)}}] For $s > 2$, there is an integer $c$ depending on $s$ such that for all $m > c$, \large 
\begin{equation}\label{conject1}
\sum_{n=m}^{2 m} \frac{1}{\log^2 (n+1)} < \sum_{n=m}^{2 m} \frac{1}{1 + \left| A_n \right| n^s} .
\end{equation}
\Large 
\end{enumerate}
\end{conjecture}

Both (a) and (b) imply the twin prime conjecture. \eqref{conject2} says that there is a pair of twin primes between $m$ and $2 m$. As $m \longrightarrow \infty$, the left of \eqref{conject1} diverges but the inequality \eqref{conject1} itself is speculation. 

\section{Circles, spheres, and calculation of $n X$}

We conclude by pointing out connections to the Chebyshev polynomials, the unit circle, and the unit sphere, and alternating groups $\mathcal{A}_4$ and $\mathcal{A}_5$. Firstly, it is useful to efficiently compute \normalsize $n X = \overbrace{X \oplus X \oplus \dots \oplus X}^{n \text{ times}} $. \Large  



\begin{lemma}\label{cheb}
Let $X = \left(
\begin{array}{cccc}
 x_1 & x_2 & x_3 & x_4 \\
\end{array}
\right) \in \mathcal{S} \left( \mathcal{R} \right)$ and define the Chebyshev polynomials $T_n(x)$ and $U_n(x)$ by 
\begin{align}
\label{firstkind} T_0 & = 1, & T_1 & = x, & T_n & = 2 x T_{n-1} - T_{n-2} , \\
\label{secondkind} U_{0} & = 1, & U_1 & = 2 x,    & U_n & = 2 x U_{n-1} - U_{n-2} .
\end{align}
 Then 
 \begin{equation}\label{nx}
 n X = \left(
\begin{array}{cccc}
 T_{n} \left( x_1 \right) & x_2 U_{n-1} \left( x_1 \right)  & x_3 U_{n-1} \left( x_1 \right) & x_4 U_{n-1} \left( x_1 \right) \\
\end{array}
\right) ,
 \end{equation}
 and 
 \normalsize
\begin{align*}
T_{2 n - 1} & =  2 T_{n} T_{n - 1} - x , & U_{2 n - 1} & = 2 \left( x U_n - U_{n-1} \right) \left( 2 x U_{n-1} - U_n \right) + 2 x , \\
T_{2 n} & =  2 T_{n}^2 - 1 , & U_{2 n} & = 2 U_{n-1} \left( x U_n-U_{n-1} \right) + 1, \\
T_{2 n + 1} & =  2 \left( 2 x T_n - T_{n-1} \right) T_n - x , & U_{2 n + 1} & = 2 U_n \left( x U_n - U_{n-1} \right) .
\end{align*}
\Large
 \end{lemma}
 
\begin{proof}
This is a consequence of two simple induction arguments. One must first prove by induction that 
 \begin{equation}\label{Chebone}
 \left(
\begin{array}{c}
 T_n \\
 U_{n-1} \\
\end{array}
\right) = \left(
\begin{array}{cc}
 x & x^2 - 1 \\
 1 & x \\
\end{array}
\right) \left(
\begin{array}{c}
 T_{n-1} \\
 U_{n-2} \\
\end{array}
\right) .
 \end{equation}
Note $0X = O$ and $1 X = \left(
\begin{array}{cccc} 
T_1 \left( x_1 \right) & x_2 U_0 \left( x_1 \right) & x_3 U_0 \left( x_1 \right) &  x_4 U_0 \left( x_1 \right) 
\end{array}
\right)$. Assume that $$(n-1) X = \left(
\begin{array}{cccc}
 T_{n-1} \left( x_1 \right) & x_2 U_{n-2} \left( x_1 \right)  & x_3 U_{n-2} \left( x_1 \right) & x_4 U_{n-2} \left( x_1 \right) \\
\end{array}
\right) .$$ Then it is easy to show that
\large 
 \begin{eqnarray*}
X \oplus (n-1) X & = & \left(
\begin{array}{cccc} 
T_n \left( x_1 \right) & x_2 U_{n-1} \left( x_1 \right) & x_3 U_{n-1} \left( x_1 \right) & x_4 U_{n-1} \left( x_1 \right) 
\end{array}
\right) .
\end{eqnarray*}
\Large
Equation \eqref{nx} follows by the principle of mathematical induction. 

The identities for $T_{2 n - 1}$, $T_{2 n}$, and $T_{2 n + 1}$ can be obtained from the functional equation 
\begin{equation*}
T_{m n}(x) = T_m \left( T_n (x) \right) = T_n \left( T_m (x) \right) ,
\end{equation*}
which follows from associativity of $\oplus$. The identities for the Chebyshev polynomials of the second kind follow from those for $T_{2 n - 1}$, $T_{2 n}$, and $T_{2 n + 1}$ together with the recursive definitions \eqref{firstkind}, \eqref{secondkind} and identity \eqref{Chebone}.
 \end{proof}
 
The Chebyshev polynomials also serve as multiplication polynomials for the group of points of the unit circle, see Figure \ref{uncirc}, so Lemma \ref{cheb} is possibly unsurprising since: 

\begin{remark}\label{conics}
The set of all points $X = \left(
\begin{array}{cccc}
 x_1 & 0 & x_3 & 0 \\
\end{array}
\right) \in \mathcal{S} \left( \mathcal{R} \right)$ forms a non-normal subgroup of $\mathcal{S} \left( \mathcal{R} \right)$ isomorphic to the group of points of the unit circle $\mathcal{C} : x_1^2 + x_3^2 = 1$ with coordinates in a commutative ring with unity $\mathcal{R}$. Letting $\mathcal{S}_2$ be the unit sphere and $p$ be an odd prime, there is a bijection $\mathcal{S}(\mathbb{Z}/p) / \mathcal{C}(\mathbb{Z}/p) \longrightarrow \mathcal{S}_2(\mathbb{Z}/p)$.
\end{remark}

\begin{lemma}
Let $H = \left\{ O, T \right\} $, where $O = \left(
\begin{array}{cccc}
 1 & 0 & 0 & 0 \\
\end{array}
\right)$ and \\
$T = \left(
\begin{array}{cccc}
 - 1 & 0 & 0 & 0 \\
\end{array}
\right)$ and let $p$ be an odd prime. Then $H$ is a normal subgroup of $\mathcal{S} \left( \mathbb{Z}/p \right)$.
\end{lemma}

\begin{proof}
We have $ X \oplus O \ominus X = O \in H$ and $ X \oplus T \ominus X = T \in H$ so $H$ is a normal subgroup of $\mathcal{S} \left( \mathbb{Z}/p \right)$.  
\end{proof}

This means that the points of the $3$-sphere with non-negative coordinates forms a group. When the coordinates are in the field $\Z/p$ for $p = 3$ or $4$, these groups are isomorphic to the alternating groups.  

\begin{example}
Let $p = 3$ and $H = \left\{ \left(
\begin{array}{cccc}
 1 & 0 & 0 & 0 \\
\end{array}
\right) , \left(
\begin{array}{cccc}
 2 & 0 & 0 & 0 \\
\end{array}
\right) \right\}$. Then $\# \mathcal{S} \left( \mathbb{Z}/p \right) = 24$, $\mathcal{S}(\mathbb{Z}/3)/H = \{ H, I, J, K, A, B, C, D, U, V, W, X \} $, 
\large 
\begin{align*}
H & = \left[ \left(
\begin{array}{cccc}
 1 & 0 & 0 & 0 \\
\end{array}
\right) \right] , & I & = \left[ \left(
\begin{array}{cccc}
 0 & 0 & 0 & 1 \\
\end{array}
\right) \right] , & J & = \left[ \left(
\begin{array}{cccc}
 0 & 0 & 1 & 0 \\
\end{array}
\right) \right] , \\
K & = \left[ \left(
\begin{array}{cccc}
 0 & 1 & 0 & 0 \\
\end{array}
\right) \right] , & A & = \left[ \left(
\begin{array}{cccc}
 1 & 1 & 1 & 1 \\
\end{array}
\right) \right] , & B & = \left[ \left(
\begin{array}{cccc}
 1 & 2 & 1 & 2 \\
\end{array}
\right) \right] , \\
 C & = \left[ \left(
\begin{array}{cccc}
 1 & 1 & 2 & 2 \\
\end{array}
\right) \right] , & D & = \left[ \left(
\begin{array}{cccc}
 1 & 2 & 2 & 1 \\
\end{array}
\right) \right] , & U & = \left[ \left(
\begin{array}{cccc}
 1 & 1 & 1 & 2 \\
\end{array}
\right) \right] , \\
V & = \left[ \left(
\begin{array}{cccc}
 1 & 1 & 2 & 1 \\
\end{array}
\right) \right] , & W & = \left[ \left(
\begin{array}{cccc}
 1 & 2 & 2 & 2 \\
\end{array}
\right) \right] , & X & = \left[ \left(
\begin{array}{cccc}
 1 & 2 & 1 & 1 \\
\end{array}
\right) \right] .
\end{align*}
\Large
It is easy to check that $\mathcal{S}(\mathbb{Z}/3)/H$ is isomorphic to the alternating group $\mathcal{A}_4$ and we have the following Cayley table for $\mathcal{S}(\mathbb{Z}/3)/H$: 
$$\begin{array}{c|cccc|cccc|cccc}
\oplus & H & I & J & K & A & B & C & D & U & V & W & X \\ 
\hline H & H & I & J & K & A & B & C & D & U & V & W & X \\
 I & I & H & K & J & B & A & D & C & V & U & X & W \\
 J & J & K & H & I & C & D & A & B & W & X & U & V \\
 K & K & J & I & H & D & C & B & A & X & W & V & U \\ \hline
 A & A & C & D & B & W & U & V & X & K & I & H & J \\
 B & B & D & C & A & X & V & U & W & J & H & I & K \\
 C & C & A & B & D & U & W & X & V & I & K & J & H \\
 D & D & B & A & C & V & X & W & U & H & J & K & I \\ \hline
 U & U & X & V & W & J & I & K & H & D & A & C & B \\
 V & V & W & U & X & K & H & J & I & C & B & D & A \\
 W & W & V & X & U & H & K & I & J & B & C & A & D \\
 X & X & U & W & V & I & J & H & K & A & D & B & C \\
 \end{array}$$
\end{example}

\begin{example}
Let $p = 5$ and $H = \left\{ \left(
\begin{array}{cccc}
 1 & 0 & 0 & 0 \\
\end{array}
\right) , \left(
\begin{array}{cccc}
 4 & 0 & 0 & 0 \\
\end{array}
\right) \right\}$. $\mathcal{S} \left( \mathbb{Z}/p \right)$ has $120$ points and $\mathcal{S} \left( \mathbb{Z}/p \right)/H$ is isomorphic to the alternating group $\mathcal{A}_5$. This pattern does not continue since $\frac{1}{2} R_4(7) \not= \frac{1}{2} n !$, $n \in \mathbb{Z}$. 
\end{example}

\section{Acknowledgments} 

I thank Franz Lemmermeyer for pointing out that $R_4(n)$ can be counted via the approach taken in Lemmas \ref{Weil}, \ref{Hensel}, and \ref{multiplicative} and for the advice to investigate the relationship between circles and spheres.

\end{document}